\documentclass[12pt]{amsart}
\usepackage{amssymb, latexsym, amscd}

\theoremstyle{plain}
\newtheorem{theorem}{Theorem}[section]
\newtheorem{corollary}[theorem]{Corollary}
\newtheorem{proposition}[theorem]{Proposition}
\newtheorem{lemma}[theorem]{Lemma}

\theoremstyle{definition}
\newtheorem{definition}[theorem]{Definition}
\newtheorem{remark}[theorem]{Remark}
\newtheorem{example}[theorem]{Example}

\numberwithin{equation}{section}

\usepackage{hyperref}
\usepackage[all]{xy}

\begin{document}

\title[Atomic pseudo-valuation Domains]{Atomic pseudo-valuation Domains}

\author{Elijah Stines}
\address{Department of Mathematics\\
         Iowa State University\\
         Ames, Iowa 50011, U.S.A.}
\email{ejstines@iastate.edu} \urladdr{http://www.public.iastate.edu/~ejstines/}

\maketitle

\begin{abstract}
   Pseudo-valuation domains have been studied since their introduction in 1978 by Hedstrom and Houston. Related objects, boundary valuation domains, were introduced by Maney in 2004. Here, it is shown that the class of atomic pseudo-valuation domains coincides with the class of boundary valuation domains. It is also shown that power series rings and generalized power series rings give examples of pseudo-valuation domains whose congruence lattices can be characterized. The paper also introduces, and makes use of, a sufficient condition on the group of divisibility of a domain to guarantee that it is a pseudo-valuation domain.

\end{abstract}


MSC: 06, 13

\maketitle

\tableofcontents

\section{Introduction}\label{S:intro}

Given an integral domain $R$, a prime ideal $I$ of $R$ is said to be strongly prime if, for every $a,b \in QF(R)$ (the quotient field of $R$), $ab\in I$ implies either $a\in I$ or $b\in I$. In their 1978 paper \cite{MR0485811}, Hedstrom and Houston investigated the class of (integral) domains where every prime ideal is strongly prime. They named these domains pseudo-valuation domains.

Recall that a domain $V$ is a valuation domain if, for every $a\in QF(V)$, either $a\in V$ or $a^{-1} \in V$. An equivalent definition of PVDs is that they are the domains $R$ which have unique valuation overrings having the same po-set of prime ideals.

\begin{definition}
    Given a domain $R$ in which every nonzero element can be factored into unique lengths of irreducibles (in other words, $R$ is a half factorial domain or HFD), $R$ is a \emph{boundary valuation domain} or BVD if every element of the quotient field of $R$ with more irreducibles on the numerator is in $R$ itself. That is to say, for every $\frac{a}{b}\in QF(R)$ with irreducible factorization $\frac{\pi_1...\pi_n}{\eta_1...\eta_m}$, with $n>m$, $\frac{a}{b}\in R$.
\end{definition}

The interplay between PVDs and their valuation overrings was examined implicitly by Maney in \cite{MR2032466}. In that paper, the class of all BVDs was characterized solely in terms of necessary and sufficient divisibility properties. The main result of the present paper is a complete characterization of all atomic PVDs in terms of their divisibility properties. It is shown that the class of atomic PVDs is precisely the class of BVDs. This characterization is an important tool for the study of the structure of domains by using their divisibility properties.

In addition, we use the investigation of the divisibility structure of PVDs to construct the lattice of all congruences for domains $R$ of a certain subclass of PVDs, those arising from restricting coefficients of power series rings. Taking this notion a step further, it is seen that many more examples of PVDs can be constructed in a similar fashion by generalizing the exponents from a power series ring.

\section{Preliminary Facts}\label{S:PVDs}

It is essential to recall some basic facts about groups of divisibility and ordered groups from \cite{MR598630} and \cite{MR720862}. It is also necessary to record some facts about PVDs proved in \cite{MR0485811} that are relevant to this investigation.

\begin{definition}
    Given a group $G$ with a partial order $\leq$ on the set $G$, we say that $G$ is a \emph{partially ordered group} (po-group) if, for all $a,b,c\in G$ with $a\leq b$, $ca\leq cb$ and $ac\leq bc$. Furthermore, given a po-group $G$, define $G^+=\{ a \in G | 0 \leq a \}$ which will be called the \emph{positive cone} of $G$.

    Po-groups $G$ and $H$ are said to be \emph{order isomorphic}, symbolically $G\cong_o H$, if there exists a group isomorphism $\phi: G\rightarrow H$ such that both $\phi$ and $\phi^{-1}$ are order preserving. It is evident that a group homomorphism $\phi$ is order preserving if and only if $\phi(G^+)\subset H^+$.
\end{definition}

\begin{definition}\label{D:groupofdivis}

 Let $R$ be a domain. We denote the subset of nonzero elements $R^\sharp$ and group of units $U(R)$. The \emph{group of divisibility} of $R$ is the quotient group $G(R):=QF(R)^\sharp / U(R)$.
For any integral domain, $G(R)$ is partially ordered by divisibility $|$, meaning $\alpha U(R) \leq \beta U(R)$ if and only if $\alpha | \beta$, that is $\exists$ $r\in R$ such that $\alpha r=\beta$.
\end{definition}

\begin{proposition} \cite{MR0485811}\label{P:localdomain}
    The set of prime ideals of a PVD $R$ is linearly ordered. As a consequence of, all PVDs are local.
\end{proposition}

The following theorem shows the importance of the unique maximal ideal $M$ of a PVD.

\begin{theorem}\cite{MR0485811} \label{T:uniqueoverring}
For a given local domain $R$ with maximal ideal $M$, the following are equivalent:
    \begin{enumerate}
    \item $R$ is a PVD;

    \item $R$ has a unique valuation overring $V$ with maximal ideal $M$;

    \item $R$ has a unique maximal ideal $M$ and $M$ is strongly prime;

    \item There exists a valuation overring $V$ in which every prime ideal of $R$ is a prime ideal of $V$.
    \end{enumerate}
\end{theorem}

\section{PVDs From Lexicographic Sums}

In this section, conditions are placed on the group of divisibility of a domain $R$, sufficient for $R$ to be a PVD. First, recall the following definitions from ring theory and the theory of po-sets.

\begin{definition}
    A subset $S$ of a ring $R$ is said to be \emph{saturated multiplicative} if it is a wall under multiplication. That is, $xy\in S$ if and only if $x \in S$ and $y\in S$.
\end{definition}

\begin{definition}
    For a po-set $X$, a subset $C$ is \emph{convex} if, whenever $a, b \in C$ with $a\leq b$, then $c\in C$ whenever $a\leq c \leq b$. Furthermore, if $X$ is a po-group a subset $C$ is \emph{directed} if every element of $C$ can be written as a difference of positive elements.
\end{definition}

\begin{definition}
    For abelian po-groups $A$ and $B$, there is a po-group $A\circ B$, called the \emph{lexicographic sum} of $A$ and $B$. The group structure on $A\circ B$ is that of $A\oplus B$, the direct sum of $A$ and $B$, with order relation $(a_1,b_1)\leq (a_2,b_2)$ if and only if $a_1 < a_2$ or $a_1=a_2$ and $b_1\leq b_2$. The other common partial order on the group $A\oplus B$ is the product order, where $(a_1,b_1)\leq (a_2, b_2)$ if and only if $a_1\leq a_2$ and $b_1\leq b_2$. The po-group with this partial order will be denoted $A\oplus B$.
\end{definition}

The following theorem, proved in \cite{MR0364213}, provides insight into the interplay between the structure of the group of divisibility of a domain and the structure of the domain itself.

\begin{theorem}\cite{MR0364213} \label{T:mottthm}
Let $R$ be a domain and $G(R)$ its group of divisibility. Then there is a one to one order reversing correspondence between the saturated multiplicative subsets of $R$ and the convex directed subgroups of $G(R)$.

\end{theorem}

The results that follow show that if the group of divisibility of a domain is a lexicographic sum of a linearly ordered group and a trivially ordered group (an antichain group), then $R$ is a PVD. It is not clear, however, if this condition on the group of divisibility is necessary for $R$ to be a PVD.

\begin{lemma}\label{L:lem1}
    If $R$ is a domain with $G(R)\cong_o L \circ A$, where $L$ is linearly ordered and $A$ is trivially ordered, then the set of convex directed subgroups is linearly ordered.
\end{lemma}

\begin{proof}
Suppose $M$ and $N$ are two convex directed subgroups of $G(R)$. Since the order isomorphism holds, we know that $M$ and $N$ correspond to two convex directed subgroups in $L \circ A$. We would like to prove that the groups $M$ and $N$ are related by inclusion. Since $M$ and $N$ are convex directed, they are generated by their positive elements. Suppose that $m \in M^+$. If $m \leq n$ for some $n \in N^+$, then $m \in N^+$, because $N^+$ is convex directed and $M \subseteq N$. Alternatively, if for every $n \in N^+$ we have $n \leq m$, then $N \subseteq M$. The only other case to consider is that there exists an $n \in N^+$ such that $n$ is incomparable to $m$, which means if $m \mapsto (l_1, a_1)$ in $L \circ A$ and $n \mapsto (l_2, a_2)$, then $l_1 = l_2$ and $a_1 \not= a_2$. So we have that $(l_1, a_1) \leq (2l_1, 2a_2)$ since $l_1 \in L$ and $l_1 \geq 0$. Thus $(l_1, a_1) \leq (2l_1, 2a_2) = (2l_2, 2a_2) = 2(l_2, a_2)$ which corresponds to $2n \in N^+$, hence $m \leq 2n$, $m \in N^+$, so $M \subseteq N$ and the set of convex directed subgroups is linearly ordered.
\end{proof}

\begin{lemma}\label{L:lem2}
    Let $R$ be as in Lemma \ref{L:lem1}. Then:
    \begin{enumerate}
    \item The set of prime ideals of $R$ is linearly ordered.
    \item The domain $R$ is local.
    \end{enumerate}
\end{lemma}

\begin{proof}
From Lemma \ref{L:lem1}, the set of convex directed subgroups of $R$ is linearly ordered. By Theorem \ref{T:mottthm}, there is a one-to-one order correspondence between the convex directed subgroups of $G(R)$ and the saturated multiplicative subsets of $R$. Thus, the saturated multiplicative subsets of $R$ are linearly ordered by set inclusion as well. If $P_1$, $P_2$ are prime ideals of $R$ and each prime ideal is the complement of a saturated multiplicative subset in $R$, we may say that $S_1 = R \smallsetminus P_1$ and $S_2 = R \smallsetminus P_2$ where $S_1$ and $S_2$ are saturated multiplicative subsets of $R$. Without loss of generality suppose that $S_1 \subseteq S_2$, so we have that $P_2 \subseteq P_1$ and the prime ideals of $R$ are linearly ordered. Thus, since $\textrm{Max}(R)$ is nonempty, $\textrm{Max}(R) = \{M\}$ for $M$ some maximal ideal of $R$. Further, we know that $M = R \smallsetminus U(R)$ because $U(R)$ is a saturated multiplicative subset. Therefore, $R \smallsetminus U(R)$ is a prime ideal of $R$ and $M \subseteq R \smallsetminus U(R)$, so from $M$ being maximal and the set of prime ideals being linearly ordered we have $M = R\smallsetminus U(R)$.

Since the set of prime ideals of $R$ is linearly ordered by set inclusion, if there were two maximal ideals, they would be comparable under set inclusion. This clearly shows that $R$ is a local domain.
\end{proof}

\begin{lemma}\label{L:lem3}
     Let $R$ be as in Lemma \ref{L:lem1}. Then the maximal ideal of $R$ is strongly prime.
\end{lemma}

\begin{proof}
 From Lemma \ref{L:lem2} we know that $R$ has a unique maximal ideal $M$. Suppose that $\alpha$,$\beta \in QF(R)$ the quotient field of $R$, such that $\alpha \beta \in M$. We want to show that $\alpha \in M$ or $\beta \in M$. First, note that $\alpha \beta= 0$ if and only if $\alpha = 0$ or $\beta = 0$ since $R$ is a domain, hence $\alpha \in M$ or $\beta \in M$. Alternatively, if $\alpha \beta \not= 0$ then since $\alpha\beta \in M$ then $\alpha \beta$ corresponds to $\alpha \beta U(R)$ in $G(R)$ and that corresponds to $(l_1 + l_2, a_1 + a_2) > (0, 0)$ in $L \circ A$. This is because all elements in $G(R)$ map to positive elements in $L\circ A$, thus $l_1+l_2 > 0$ so $l_1 > -l_2$ so if $l_1 > 0$ then $\alpha \in M$ otherwise $l_1 \leq 0$ means that $l_2 > 0$ and $\beta\in M$ and thus $\alpha \in M$ or $\beta \in M$ and $M$ is strongly prime and $R$ is a PVD.
\end{proof}

\begin{theorem} \label{T:secondthm}
If $R$ is a domain with $G(R)\cong_o L \circ A$, where $L$ is linearly ordered and $A$ is trivially ordered, then $R$ is a PVD.
\end{theorem}

\begin{proof}
    From Lemma \ref{L:lem2} we have that $R$ must have a unique maximal ideal. From Lemma~\ref{L:lem3} we see that this maximal ideal must be strongly prime. From Theorem~\ref{T:uniqueoverring}, this is enough to show that $R$ is a PVD.

\end{proof}

Integral domains $R$ where $G(R)$ satisfies the conditions of Theorem \ref{T:secondthm} are quite common. Section \ref{S:lattices} contains many different kinds of examples. Further, all valuation domains are included in this collection, as is the class of all BVDs.

\section{A Particular Class of PVDs}\label{S:lattices}

In this section, the congruence lattices of rings from a class of PVDs are characterized. The goal of this section is to represent the lattice of congruences of integral domains of the form $K+XF[[X]]$, where $K\subseteq F$ is a field extension.

\begin{example}\label{E:baseexample} Let $K\subseteq F$ be a field extension. Then $R=K+XF[[X]]$ is a PVD. This is because $V=F[[X]]$ is an overring of $R$ and $V$ is a valuation domain. Also, $R$ and $V$ are local with maximal ideal $\langle X \rangle$. So $G(R) \cong_o G(V) \circ U(V)/U(R)$. However $G(V) \cong_o \mathbb{Z}$, $U(V) \cong F^\sharp$, and $U(R) \cong K^\sharp$. So, we have that $G(R) \cong_o \mathbb{Z} \circ F^\sharp/K^\sharp$.
\end{example}

Taking Example \ref{E:baseexample}, examining a specific field extension $K\subseteq F$, and constructing the entire lattice helps motivate a more general method to handle all field extensions.

\begin{example}\label{E:specificexample} Consider $F = \mathbb{F}_4=\{0,1,a,b\}$, the four element field, and $K=\{0,1\}$. Then from Example \ref{E:baseexample}, $G(R) \cong_o \mathbb{Z} \circ \{1,a,b\}/\{1\}$, or equivalently $G(R) \cong_o \mathbb{Z} \circ \{1,a,b\}$, where the second factor is ordered trivially. Thus, the Hasse diagram in Figure \ref{F:fig1} extended below and above for every element of $\mathbb{Z}$ represents the dual of the entire po-set of principal ideals of $R$.

\begin{figure}
    $$\xymatrix{(n,1)\ar@{-}[d]\ar@{-}[rd]\ar@{-}[rrd] & (n,a)\ar@{-}[d]\ar@{-}[ld]\ar@{-}[rd] & (n,b)\ar@{-}[d]\ar@{-}[ld]\ar@{-}[lld] \\ (n-1,1) & (n-1,a) & (n-1,b) } $$
\caption{A portion of the divisibility relation on $\mathbb{F}_2+X\mathbb{F}_4[[X]]$}\label{F:fig1}
\end{figure}
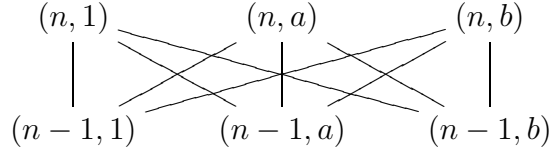

As in any integral domain, the dual of the positive cone of $G(R)$ represents the order relationship between the principal ideals. Therefore the principal ideals of $R$ correspond to the Hasse diagram in Figure \ref{F:fig2}.

\begin{figure}
    $$\xymatrix@R=.25cm{& R\ar@{-}[ld]\ar@{-}[d]\ar@{-}[rd] & \\(1,1)& (1,a)& (1,b)\\& \vdots &\\(n,1)\ar@{-}[rd]\ar@{-}[d]\ar@{-}[rrd] & (n,a)\ar@{-}[d]\ar@{-}[ld]\ar@{-}[rd] & (n,b)\ar@{-}[d]\ar@{-}[ld]\ar@{-}[lld] \\ (n+1,1) & (n+1,a) & (n+1,b)\\ & \vdots & \\ & \langle 0 \rangle&}$$
    \caption{Hasse diagram of principal ideals of $\mathbb{F}_2+X\mathbb{F}_4[[X]]$}\label{F:fig2}
\end{figure}
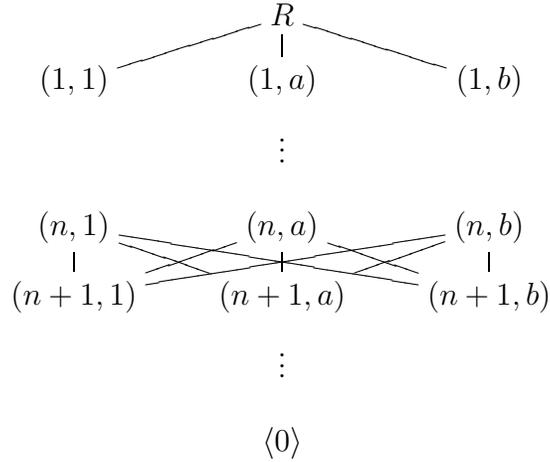

Now, all that is needed to complete the Hasse diagram of the congruences of $R$ is the meets and joins of the principal ideals. This is trivial by looking at a specific natural number and computing the meets and joins there. So, consider the three ideals corresponding to the generators $(n,1)$,$(n,a)$, and $(n,b)$. These ideals are generated by power series with lowest term $X^n$, $aX^n$, and $bX^n$ respectively. Therefore, since the joins are taken by linear combinations of the generators of these principal ideals, we have that $(n,1)\vee (n,a)=(n,a) \vee (n,b)=(n,1) \vee (n,b) = \langle \mathbb{F}_4 X^n \rangle$. This ideal may be identified by just $\langle X^n \rangle_V$ where the subscript denotes
the fact that the generation of the ideal takes place in the valuation overring.

It is also evident that $(n,1) \wedge (n,a) = (n,1) \wedge (n,b) = (n,a) \wedge (n,b) = \langle X^{n+1} \rangle_V$. So finally the Hasse diagram of the congruence lattice of $R$ is as in Figure \ref{F:fig3}.

\begin{figure}
    $$\xymatrix@R=.25cm{& R\ar@{-}[d] & \\ & \langle X \rangle_V\ar@{-}[ld]\ar@{-}[d]\ar@{-}[rd] & \\(1,1)\ar@{-}[rd]& (1,a)\ar@{-}[d]& (1,b)\ar@{-}[ld]\\& \vdots \ar@{-}[d]&\\ & \langle X^n \rangle_V \ar@{-}[ld]\ar@{-}[d]\ar@{-}[rd]& \\ (n,1)\ar@{-}[rd]& (n,a)\ar@{-}[d]& (n,b)\ar@{-}[ld]\\ & \langle X^{n+1} \rangle _V\ar@{-}[ld]\ar@{-}[d]\ar@{-}[rd]& \\(n+1,1)\ar@{-}[rd]& (n+1,a)\ar@{-}[d]& (n+1,b)\ar@{-}[ld]\\ & \vdots  \ar@{-}[d]& \\ & \langle 0 \rangle&}$$
    \caption{Congruence lattice of $R$}\label{F:fig3}
\end{figure}
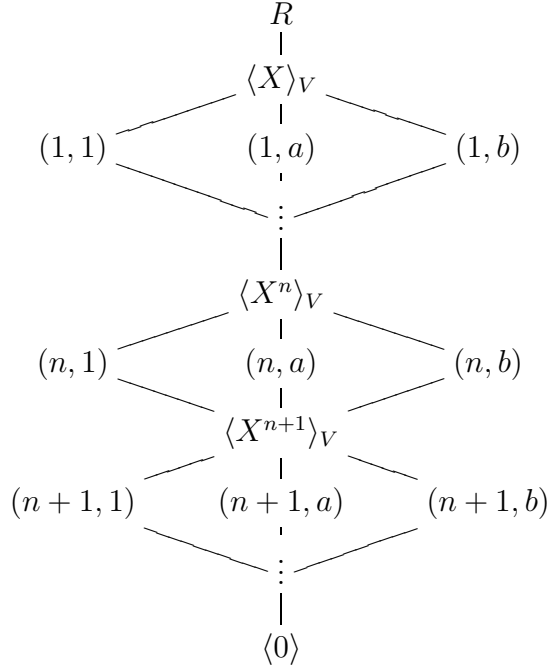

As one may see, even the simplest examples of these constructions of PVDs have congruence lattices that are highly non-distributive.

\end{example}

The motivation for the generalization of this example to arbitrary field extensions comes from the construction of the meets and the joins of the principal ideals. Since the joins in particular can be expressed as linear combinations (over $K$) of the generators of the two previous ideals, this invites us to more thoroughly investigate the vector space congruence structure of $\mathrm{AG}(F,K)$ (where $K\subseteq F$), in order to generate the non-principal ideals of $K+XF[[X]]$.

\begin{definition}
    For a po-set $P$ with order relation $\leq$, we define the \emph{dual poset} $P^\partial$ to be the poset over the set $P$ with order relation $\leq_\partial$, where $a\leq_\partial b$ if and only if $b\leq a$.
\end{definition}

\begin{definition}
    Let $K \subseteq F$ be a field extension. Then the poset of all nonzero subspaces of $F$ as a vector space over $K$ is denoted $\mathrm{AG}(F,K)$.
\end{definition}

\begin{theorem}

Let $K\subseteq F$ be a field extension, then the lattice of ideals of the power series ring $R=K+XF[[X]]$ is lattice isomorphic to $1\oplus (\mathbb{N}^\partial \circ \mathrm{AG}(F,K))\oplus 1$, where $1$ is the one element lattice.

\end{theorem}

\begin{remark}
The structure of the above lattice can be decomposed into three distinct parts:

 \begin{enumerate}
    \item the first one-element lattice, representing the ideal $\langle 0\rangle$;
    \item the lexicographic produce which first indexes the lowest power occurring on the indeterminate inside the ideal and then compares according to the number of generators required to create the available leading coefficients;
    \item the last one-element lattice represents the entire ring.
 \end{enumerate}

It is certainly true that lexicographic products of po-sets do not always result in lattices. In this case, however, identifying the full vector space of $F$ over $K$ at a given coordinate $n$ with the trivial subspace in the next higher coordinate $n-1$, it is as if the lattice of subspaces of $F$ over $K$ were repeated once for each value $n\in \mathbb{N}$, resulting in a lattice.
\end{remark}

\begin{proof}
The unique valuation overring of $R$ is $V=F[[X]]$ which has positive cone of divisibility order isomorphic to $\mathbb{N}$. Also, recall, as a consequence of Example \ref{E:baseexample}, that the poset of nonzero elements of $R$ ordered under divisibility $R^\sharp/U(R)$ is order isomorphic to $V^\sharp/U(V)\circ U(V)/U(R)$. Therefore, when identifying a principal ideal of $R$ we may refer to an ordered pair $(n,\alpha)$, where $n \in \mathbb{N}$ and $\alpha \in U(V)$ is a coset representative for $\alpha U(R)$. Of course, the largest principal ideal is $(0,1)$, which corresponds to the
entire ring, and the smallest is $\langle 0 \rangle$, since the principal ideals are ordered dually to the lattice of divisibility.

Let $L$ be the lattice of ideals of $R$. Then any nonzero ideal $I \in L$ can be expressed as a join of principal ideals $I =\bigvee\limits_{\lambda\in \Lambda} (n_\lambda, \alpha_\lambda)$ for an indexing set $\Lambda$. Recall, that any principal ideal $(n,\alpha)\subseteq (m, \beta)$ if and only if $m>n$ or $m=n$ and $\alpha\beta^{-1} \in U(R)$. So, let $N= \inf\limits_{\lambda \in \Lambda} n_\lambda$, which exists since the ideal generated is not the zero ideal. Then $I = \bigvee\limits_{\sigma\in\Sigma} (N, \alpha_\Sigma)$ for $\Sigma=\{\lambda\in \Lambda |$ $ n_\lambda=N\}$.

At this point, the join of the principal ideals involved can be expressed as their sum up to multiples from $R$. Since each principal ideal contains all series with degree lower than $N$ we need only consider unit multiples (from $R$) acting on the generating set. So, in reality
$$I=\bigvee\limits_{\sigma \in \Sigma} (N,\alpha_\sigma)= \bigg\langle x^N\alpha \bigg| \alpha \in \sum\limits_{\sigma \in \Sigma} \alpha_\sigma \beta_\lambda \text{ and } \beta \in K \bigg\rangle $$

Thus, the lattice of ideals $L$ is isomorphic to Figure \ref{F:fig4} where $V_k$ are copies of the poset of nontrivial subspaces of the vector space $\mathrm{AG}(F,K)$ where the elements generated in that subspace show up as the available coefficients on the lowest term of the series of degree $k$. In short $L \cong 1 \oplus (\mathbb{N}^\partial \circ \mathrm{AG}(F,K))\oplus 1$ (in the category of lattices) as desired.
\end{proof}

\begin{figure}
    $$\xymatrix@R=.25cm{&R\ar@{-}[d]&\\&V_1\ar@{-}[d]&\\&V_2\ar@{-}[d]&\\& \vdots\ar@{-}[d]&\\&\langle 0 \rangle&}$$
    \caption{Basic structure of congruence lattice}\label{F:fig4}
\end{figure}
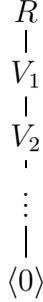

This section concludes with one final example which examines the entire lattice of ideals of a PVD of type $K+XF[[X]]$.

\begin{example}
Let $R= \mathbb{Q}+ X (\mathbb{Q}[\sqrt[3]{2}])[[X]]$. It is impossible to give an entire description of the lattice of ideals of $R$ since $\mathrm{AG}(\mathbb{Q}[\sqrt[3]{2}], \mathbb{Q})$ is infinite, but it is possible to model what happens with selected elements. We can observe this in the lattice of Figure \ref{F:fig5}.

\begin{figure}
    $$\xymatrix@R=.25cm{& R\ar@{-}[d] & \\
    & \langle X \rangle_V\ar@{-}[ld]\ar@{-}[d]\ar@{-}[rd] & \\
    \langle(1,1),(1, \sqrt[3]{2}) \rangle\ar@{-}[d]\ar@{-}[rd] & \langle (1,1), (1,\sqrt[3]{4}) \rangle \ar@{-}[ld]\ar@{-}[rd] & \langle(1, \sqrt[3]{4}) , (1,\sqrt[3]{2}) \rangle \ar@{-}[ld]\ar@{-}[d] \\
    (1,1)\ar@{-}[rd]& (1,\sqrt[3]{2}) \ar@{-}[d] & (1, \sqrt[3]{4}) \ar@{-}[ld] \\
    & \langle X^2 \rangle_V\ar@{-}[ld] \ar@{-}[d]\ar@{-}[rd] & \\
    \vdots\ar@{-}[rd] & \vdots\ar@{-}[d]& \vdots\ar@{-}[ld] \\
    & \langle 0 \rangle &}$$
    \caption{Congruence lattice of $\mathbb{Q}+ X (\mathbb{Q}[\sqrt[3]{2}])[[X]]$}\label{F:fig5}
\end{figure}
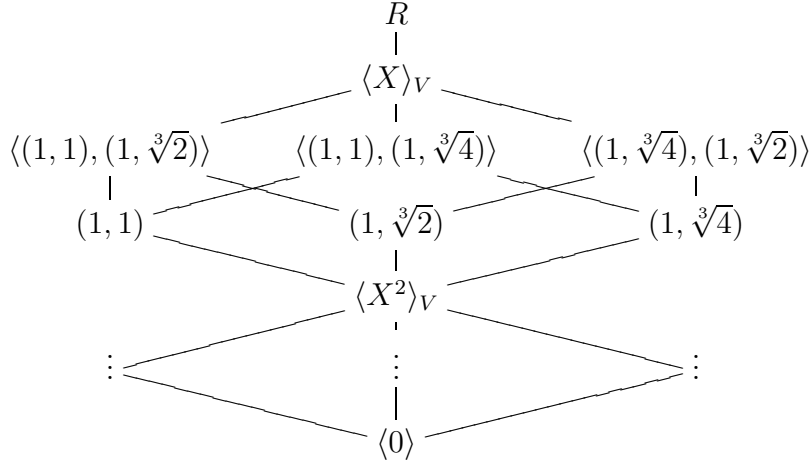

This method can be easily generalized to any field extension, but becomes very difficult to realize as the degree and the complexity of the field extension increases.
\end{example}

\section{PVDs from Generalized Power Series}\label{S:powerseries}

In the previous section, a source of PVDs was obtained from restricting the leading coefficients of power series over a field. Another source of PVDs is found by restricting the generalized power series from Ribenboim's \cite{MR1289092}. This method allows one to construct many PVDs which are not necessarily atomic, as is exhibited.

\begin{definition}
    A commutative monoid $M$ on a po-set under $\leq$ is called a \emph{linearly ordered monoid} if $\leq$ is a total order and if, for every $a,b,c \in M$, with $a\leq b$, $ac\leq bc$.
\end{definition}

\begin{definition}
    A po-set $X$ is said to be \emph{narrow} if each induced antichain is finite. A po-set $X$ is said to be \emph{Artinian (Noetherian)} if there are no infinite decreasing (increasing) sequences in $X$.
\end{definition}

Generalized power series over rings $R$ are given by specifying a partially ordered monoid $M$. The usual power series are given by considering generalized power series over rings with monoid $\mathbb{N}$. The definition is as follows.

\begin{definition}
    Given a ring $R$ and a partially ordered monoid $M$, the \emph{generalized power series ring} denoted $A=R[[M]]$ is the collection of functions $f:M \rightarrow R$ with support on a narrow, Artinian subset of $M$. Addition is given by $(f+g)(m)=f(m)+g(m)$ and multiplication given by the convolution $(f*g)(m)=\sum\limits_{m_1\in M}f(m_1)g(m-m_1)$. If, in addition, $M$ is linearly ordered, the support of each function is well ordered and we may define the minimum of the support, denoted $\min(f)$.

    Given an element $m\in M$, we may define the \emph{delta function centered at m} as $\delta_m (x):=\delta_{mx}$, the Kronecker delta.
\end{definition}

The units of generalized power series have been characterized in a similar fashion to those of classical power series. The theorem is stated for the special case that $R$ is a field.

\begin{proposition}\cite{MR1289092}
    Let $A=R[[M]]$ be a generalized power series where $R$ is a field. Then $U(A)=\{ f(m) | f(0)\not= 0\}$.
\end{proposition}

Of particular interest at present is when the partially ordered monoid $M$ is the positive cone of a linearly ordered group $G$. Generalized power series rings of this type have been studied for some time and the following result, due to Hahn in \cite{MR0171864}, is the standard example for a class of rings showing that the classification in terms of a group of divisibility has at least one ring for each isomorphism class of po-groups.

\begin{theorem}\cite{MR0171864}
    The generalized power series ring $R=F[[\Gamma^+]]$ over a field $F$ and linearly ordered group $\Gamma$ has group of divisibility $G( R)\cong_o \Gamma$.
\end{theorem}

\begin{proof}
    Given any element $f\in R^\sharp$ we have that $f= (\delta_{\min(f)})*u$ where $u$ is a unit of $R$. So, given any element $\frac{f}{g}U(R)\in QF(R)^\sharp/U(R)$, we may write $\frac{f}{g}U(R)$ as $\frac{
    \delta_{\min(f)}}{\delta_{\min(g)}}U(R)$.

    Additionally, $\frac{\delta_{\min(f)}}{\delta_{\min(g)}}\in R$ if and only if $\min(g)\leq \min(f)$. This is due to the fact that a delta function convolved with another function $f$ simply translates the support of $f$ by the support of the delta function. Therefore, the convolution of a delta function with another function $f$ is again a delta function if and only if $f$ is a delta function.

    We now consider the map $$\phi: G(R)\rightarrow \Gamma;  \frac{\delta_{\min(f)}}{\delta_{\min(g)}}U(R) \mapsto \min(f)-\min(g)$$
    It must be shown that $\phi$ is an isomorphism of po-groups. We know that $u\in U(R)$ if and only if $\min(u)=0$. Therefore $\phi$ is well defined on cosets. Further, since $\min:R\rightarrow \Gamma^+$ is an order preserving group homomorphism, we have that $\phi(R/U(R))\subseteq \Gamma^+$ so $\phi$ is order preserving and $\phi$ is a group homomorphism when extended to $G(R)$. Since $\phi(\delta_\gamma U(R))=\gamma$ and $\phi(\frac{1}{\delta_\gamma}U(R))=-\gamma$, $\phi$ is surjective. Finally, $\phi^{-1}$ is order preserving, since $\phi^{-1}(\gamma)=\delta_\gamma U(R)$ for all $\gamma\in \Gamma^+$.
\end{proof}

\begin{definition}\label{D:respowerseries}
    Let $K\subseteq F$ be a field extension. Let $\Gamma$ be a linearly ordered group with positive cone $\Gamma^+$. The subring $S$ of $R=F[[\Gamma^+]]$ consisting of functions $f$, for which $f(0)\in K$ is called the \emph{series ring} over $K\subseteq F$, with exponents in $\Gamma^+$.
\end{definition}

\begin{lemma}\label{L:units}
    The group of units of a series ring $S$ over $K\subseteq F$ is $U(S)=U(R)\cap S$.
\end{lemma}

\begin{proof}
    Suppose that $f\in S$ is a unit of $S$. This means there exists an element of $S$ so that $f*g=1\in S\subseteq R$, which means $f\in U(R)$. Now suppose that $f\in U(R)\cap S$. This means $f(0)\in F^\sharp\cap K$ which means $f$ has an inverse in $R$ whose component at $0$ is in $K$. Thus $g^{-1}\in S$ as well.
\end{proof}

\begin{theorem}\label{T:respowerseries}
    Let $S$ be a series ring over $K\subseteq F$, with exponents in $\Gamma^+$. Then $G(S)\cong_o \Gamma \circ F^\sharp/K^\sharp$.
\end{theorem}

\begin{proof}
    Observe that every element of $S^\sharp$ may be written as $$f(\min(f))\delta_{\min(f)} * u$$ where $u\in U(S)$. It must be noted that $f(\min(f))\in F^\sharp$ and not necessarily in $K^\sharp$. In the case where $f(\min(f))\in K^\sharp$, $$f(\min(f))\delta_{\min(f)} * u = \delta_{\min(f)}*v$$

    Elements of the group of divisibility $G(S)$ are therefore of the form $\frac{ f(\min(f))\delta_{\min(f)}}{g(\min(g))\delta_{\min(g)} }U(S)$. This notation is shortened for the rest of the proof by identifying such an element with $\frac{\alpha \delta_{f^0}}{\beta \delta_{g^0}}U(S)$, where $\alpha , \beta\in F^\sharp$ and $f^0:=\min(f), g^0:=\min(g)\in \Gamma^+$.

    Define $$\psi: G(S)\rightarrow \Gamma \circ F^\sharp/K^\sharp;  \frac{\alpha \delta_{f^0}}{\beta \delta_{g^0}}U(S) \mapsto (f^0-g^0, \frac{\alpha}{\beta}K^\sharp)$$

    It must be shown that $\psi$ is an isomorphism of po-groups. Observe first that $\psi$ is well defined on cosets. This is because $\psi( U(S)) = (0, K^\sharp)=(f^0-g^0, \frac{\alpha}{\beta}K^\sharp)=\psi( \frac{\alpha \delta_{f^0}}{\beta \delta_{g^0}})$ if $\frac{\alpha}{\beta}\in K^\sharp$ and $f^0-g^0=0$. To see that $\psi$ is an abelian group homomorphism, let $\frac{\alpha_1 \delta_{f^0_1}}{\beta_1 \delta_{g^0_1}}U(S), \frac{\alpha_2 \delta_{f^0_2}}{\beta_2 \delta_{g^0_2}}U(S)\in G(S)$, then $$\psi(\frac{\alpha_1 \delta_{f^0_1}}{\beta_1 \delta_{g^0_1}}\frac{\alpha_2 \delta_{f^0_2}}{\beta_2 \delta{g^0_2}}U(S))$$

    $$=(f^0_1+f^0_2-g^0_1-g^0_2 , \frac{\alpha_1\alpha_2}{\beta_1\beta_2}K^\sharp)=(f^0_1-g^0_1,\frac{\alpha_1}{\beta_1}K^\sharp)+(f^0_2-g^0_2,\frac{\alpha_2}{\beta_2}K^\sharp)$$ which is clearly the sum of the images of the individual factors under $\psi$. It is easy to see that $\psi$ preserves inverses.

    Observe that $\psi$ is surjective by taking any $\gamma\in \Gamma$ and any $\alpha K^\sharp\in F^\sharp/K^\sharp$ and writing it as the image $\psi(\alpha \delta_\gamma)U(S)$ if $\gamma\in \Gamma^+$, and \newline $\psi(\frac{\alpha\delta_0}{\delta_\gamma}U(S))$ if $-\gamma\in \Gamma^+$. Since $\Gamma$ is linearly ordered, these are the only possibilities.

    It must be shown that $\psi$ is an isomorphism of abelian groups. If there were an element $\frac{\alpha \delta_{f^0}}{\beta \delta_{g^0}}U(S)$ such that $\psi(\frac{\alpha \delta_{f^0}}{\beta \delta_{g^0}}U(S))=(0, K^\sharp)$, then $f^0=g^0$ and $\frac{\alpha}{\beta}\in K^\sharp$ by the definition of $\psi$. That is to say, the minima of the supports of $f$ and $g$ were equal and the ratio of the values at that location was an element of $K^\sharp$. Thus, $\frac{\alpha \delta_{f^0}}{\beta \delta_{g^0}}U(S)=U(S)$ by Lemma \ref{L:units}.

    To show that $\psi$ and $\psi^{-1}$ preserve order, it is sufficient to show that they preserve the positive cones of the po-groups. First, considering $\alpha \delta_{f^0} U(S)\in G(S)^+=S^\sharp/U(R)$, we have that $\psi(\alpha \delta_{f^0} U(S))=(f^0, \alpha K^\sharp)$. Since $\alpha \delta_{f^0} \in S^\sharp$, $f^0\not= 0$ and $f^0\in \Gamma^+$ or $f^0=0$ and $\alpha \in K^\sharp$, in either case $(f^0,\alpha K^\sharp)$ is in the positive cone of $\Gamma\circ F^\sharp/K^\sharp$.

    Now suppose that $(\gamma, \alpha K^\sharp)$ is in the positive cone of $\Gamma\circ F^\sharp/K^\sharp$. That is, either $\gamma\not=0$ and $\gamma\in \Gamma^+$ or $\gamma=0$ and $\alpha\in K^\sharp$. It is easily seen that $\phi^{-1}(\gamma, \alpha K^\sharp)=\alpha \delta_\gamma U(S)\in R^\sharp/U(S)$.

\end{proof}

\begin{corollary}
    Let $S$ be a series ring over $K\subseteq F$, with exponents in $\Gamma^+$. Then $S$ is a PVD.
\end{corollary}

\begin{proof}
    This is a consequence of Theorem \ref{T:respowerseries} and Theorem \ref{T:secondthm}.
\end{proof}

This section concludes with some examples of PVDs coming from series rings over $K\subseteq F$.

\begin{example} Observe the following PVDs obtained by specifying a coefficient field extension and a positive cone of an abelian po-group:

    \begin{enumerate}
        \item The (nonatomic) PVDs with $F=\mathbb{R}$, $K=\mathbb{Q}$, and $\Gamma=\mathbb{Z}\circ\mathbb{Z}$. These domains behave similarly to Laurent series domains over $\mathbb{R}$ with constant terms in $\mathbb{Q}$ and variables $X$ and $Y$ where $Y$ may have negative exponents.
        \item Series rings over field extensions $K\subseteq F$ with $\Gamma=\mathbb{R}$ are another class of nonatomic PVDs. These domains may be thought of as formal power series with restricted leading coefficients where the powers on the indeterminants are allowed to be any nonnegative real number.
        \item The only atomic examples of generalized restricted power series with $\Gamma$ linearly ordered are the standard restricted power series rings with $\Gamma=\mathbb{Z}$. This is a consequence of the next section.
    \end{enumerate}

\end{example}

\section{Classification of Atomic PVDs}

In this section, Theorem \ref{T:secondthm} is used along with Maney's classification of BVDs in \cite{MR2032466} to give several equivalent conditions for a domain $R$ to be an atomic PVD. There are many nonatomic PVDs, for example as exhibited in the previous section.It is also seen that the assumption of atomicity on a PVD $R$ implies $R$ is an HFD.

\begin{definition}
    Given an HFD $R$ with quotient field $QF(R)$, an \emph{overring} of $R$ is any ring $T$ such that $R\subseteq T \subset QF(R)$. An overring $T$ is \emph{boundary positive} if every element $x\in T^\sharp$ has at least as many irreducibles of $R$ on the numerator as the denominator. We say that $T$ is \emph{boundary complete} if, for every $x\in T^\sharp$ with an equal amount of irreducibles on the numerator and denominator, we have $x \in R^\sharp$, which is equivalent to $x\in U(R)$.
\end{definition}

Of particular interest at this point is the the classification of BVDs in terms of their groups of divisibility, proven in \cite{MR2032466}.

\begin{theorem}\label{T:Maney}
    Let $R$ be a domain with complete integral closure $R'$. Then $R$ is a BVD if and only if $G(G)\cong_o \mathbb{Z}\circ U(R')/U(R)$.
\end{theorem}

The main result of this paper is the characterization of the class of atomic PVDs as domains whose group of divisibility is an element of a certain isomorphism class of po-groups. This kind of result has precedents in the literature. As previously mentioned, it was proven by Hahn and referred to in \cite{MR0171864} that $V$ is a valuation domain if and only if $G(V)\cong_o L$ for some linearly ordered group $L$. Two other classic results referred to in \cite{MR720862} are that $R$ is a UFD if and only if $G(R)\cong_o \bigoplus\limits_{p\in \mathcal{P}}\mathbb{Z}$ with $\mathcal{P}$ the set of prime elements and sum ordered under the cardinal order. The other result is that $R$ is a GCD domain if and only if $G(R)$ is a lattice-ordered group.

\begin{theorem}
    For an integral domain $R$, the following are equivalent:
    \begin{enumerate}
        \item $R$ is a BVD;
        \item $R$ is an HFD with boundary positive, boundary complete, valuation overring $V$ with $G(V)\cong_o \mathbb{Z}$;
        \item $G(R)\cong_o \mathbb{Z} \circ U(V)/U(R)$ for some overring $V$ of $R$;
        \item $R$ is an atomic PVD.
    \end{enumerate}
\end{theorem}

\begin{proof}
     $(1) \Rightarrow (2)$ For this implication we refer to Theorem \ref{T:Maney} with the overring $V=R'$, the complete integral closure of $R$.

     $(2)\Rightarrow (3)$ It must be shown that the group of divisibility $G(R)\cong_o \mathbb{Z}\circ U(V)/U(R)$ for an overring $V$ of $R$. The overring used is the one provided by assuming $(2)$. Since $V$ is an overring of $R$, it is evident that the quotient field of $V$ is the same as that of $R$. So, any element $\frac{\alpha}{\beta}\in QF(R)$ may be written as $\frac{z^mu}{z^nv}$ where $z$ is the unique prime of the rank-1 DVR overring $V$, $n, m \in \mathbb{N}$, and $u,v\in U(V)$.

     Construct set map $$\phi: G(R)\rightarrow \mathbb{Z} \circ U(V)/U(R); \frac{z^mu}{z^nv}U(R)\mapsto (m-n, uv^{-1}U(R))$$ This is a homomorphism of abelian groups. First, suppose that there are two coset representatives for an element, that is, \newline $\frac{z^mu}{z^nv}U(R)=\frac{z^muu'}{z^nv}U(R)$. These elements map to the same pair \newline $(m-n, uu'v^{-1}U(R))$ since $U(R)\subseteq U(V)$. This homomorphism also preserves the operation of multiplication of quotient field elements as $\frac{z^{m_1}u_1}{z^{n_1}v_1}\frac{z^{m_2}u_2}{z^{n_2}v_2}=\frac{z^{m_1+m_2}u_1u_2}{z^{n_1+n_2}v_1v_2}$. The kernel of $\phi$ is the set of all $\frac{z^{m}u}{z^{n}v}\in QF(R)$ such that $m=n$, and $uv^{-1}\in U(R)$, which is simply the set $U(R)$. Finally, $\phi$ is surjective because, for every $(m,vU(R))\in \mathbb{Z}\circ U(V)/U(R)$ we may write this as $\phi(z^mvU(R))$ if $m\geq 0$, or $\phi(\frac{v}{z^{-m}}U(R))$ if $m<0$.

     To conclude this implication it must be shown that the isomorphism is, in fact, an order isomorphism of abelian groups. It is sufficient to show the preservation of positive cones. Suppose that $\frac{z^mu}{z^nv}U(R)$ is positive. that means $\frac{z^mu}{z^nv}\in R$. That means $m\geq n$ since, if not, $\frac{z^mu}{z^nv}\not\in V$ and $R\subseteq V$. If $m>n$, then $z^{m-n}uv^{-1}\in V$ is a non zero non unit and is thus in $R$. If, on the other hand, $m=n$, to be in $R$ means $uv^{-1}$ is in $R$ and its inverse $u^{-1}v$ is also in $R$, which shows that $\phi$ is order preserving. If $(n, uU(R))$ is a positive element of $\mathbb{Z}\circ U(V)/U(R)$, $n>0$ in which case $(n, uU(R))=\phi(z^nuU(R))$ and $z^nu\in R$ or $n=0$ and $u\in U(R)$, this means $(n,uU(R))=\phi(uU(R))$ and $u\in R$. This finishes the verification that $\phi$ is, in fact, an isomorphism of po-groups.

     $(3) \Rightarrow (4)$ Since $G(R)\cong_o L \circ A$, Theorem \ref{T:secondthm} states that $R$ is a PVD. We need only show that $R$ is atomic. Let $r\in R^\sharp$. It must be shown that we may write $r=ux_1x_2...x_n$ where $u\in U(R)$ and $x_i$ are irreducibles. Since $r\in R^\sharp$, the element $rU(R)$ is in the positive cone of $G(R)$. This means $\phi(rU(R))=(n,uU(R))=\sum\limits_{k=1}^{n-1} (1,U(R)) + (1, uU(R))$. Furthermore, since $\phi$ is an isomorphism of po-groups and each $(1,U(R))$ and $(1,uU(R))$ is a minimal positive element of $\mathbb{Z}\circ U(V)/U(R)$, $\phi^{-1}(\sum\limits_{k=1}^{n-1} (1,U(R)) + (1, uU(R)))$ is a product of minimal positive elements of $G(R)$, which equal $rU(R)$. Thus, $r$ may be written as a product of minimal positive elements in $G(R)$, which means that $R$ is atomic.

     $(4) \Rightarrow (1)$ Suppose that $R$ is an atomic PVD. To see that $R$ is a BVD it is sufficient to show that, for any $\frac{x}{y} U(R)\in G(R)$ with $\partial_R(\frac{x}{y})\not= 0$, we have either $\frac{x}{y}$ or $\frac{y}{x}\in R$. Since $R$ has a unique valuation overring $V$, we have, without loss of generality, $\frac{x}{y}\in V$. The only way that $\frac{x}{y}\in V\smallsetminus R$ is if $\frac{x}{y}\in U(V)\smallsetminus U(R)$, since $V\smallsetminus R= U(V)\smallsetminus U(R)$. Suppose that, in fact, $\frac{x}{y}\in U(V)\smallsetminus U(R)$. This means $\partial_V(\frac{x}{y})=0$. But, since $V$ and $R$ share the same unique maximal ideal $M$, if $\partial_V(\frac{x}{y})=0$, then $\partial_R(\frac{x}{y})=0$, contradicting our assumption. Therefore $R$ is a BVD.

\end{proof}

\bibliographystyle{plain}
\bibliography{EliDivisibility1}

\end{document}